\documentclass[11pt]{article}
\usepackage{amsmath,amssymb,amsxtra, amsfonts, amsthm}
\usepackage{footnote}

\usepackage{graphicx,xcolor,tikz}

\usepackage[hmargin=0.16 \paperwidth,vmargin=0.12\paperwidth,bindingoffset=0cm]{geometry}

\usepackage{titlesec}
\titleformat*{\section}{\large\bfseries}
\titleformat*{\subsection}{\bfseries}

\theoremstyle{plain}
  \newtheorem{thm}{Theorem}

  \newtheorem{prop}[thm]{Proposition}
\theoremstyle{definition}
  \newtheorem{remark}{Remark}
 
\numberwithin{equation}{section}

\def\be{\begin{equation}}
\def\ee{\end{equation}}

\DeclareMathOperator{\tr}{Tr} 
\DeclareMathOperator{\sgn}{sgn} 

\usepackage{hyperref}   
\hypersetup{
colorlinks=true,
linkcolor=blue,
citecolor=blue,
}

\begin{document}

%
%

\title{\bfseries\huge 
Spectral statistics for product matrix ensembles of Hermite type with external  source}
\author{\huge Dang-Zheng Liu\footnote{ Key Laboratory of Wu Wen-Tsun Mathematics, CAS, 
 School of Mathematical Sciences, University of Science and Technology of China, Hefei 230026, 
P.R.~China.  Email: dzliu@ustc.edu.cn.   Current address: Institute of Science and Technology Austria,  Klosterneuburg 3400, Austria }}

\date{\today}

\maketitle
\begin{abstract}
\noindent
We continue investigating  spectral properties of a Hermitised random matrix product
, which,  contrary to previous product ensembles, allows for eigenvalues on the full real line. When a GUE matrix with an external source is involved,  we prove that  the eigenvalues of the product form a determinantal point process and derive   a double integral representation  for  correlation kernel.  As the source changes,  we observe a critical value and establish    the existence of  a phase transition for scaled eigenvalues at the origin.  Particularly in  the critical case, we obtain a new family of Pearcey-type kernels.  \end{abstract}

\section{Introduction and main results} \label{sect:intro}

\subsection{Motivations}

In this paper we continue the investigation of  spectral properties of  Hermitised product matrix ensembles initiated in \cite{FIL17}.  More specifically,  
suppose that  (1) each $G_i$ ($i=1,\dots, M$) is a standard  complex Ginibre   matrix of size $(\nu_{m-1}+n) \times (\nu_{m} +n)$
  with $\nu_0=0, \nu_1, \ldots, \nu_M\geq 0$, i.e. a matrix with i.i.d. standard complex Gaussian entries; (2)  $H$ is  an $n\times n$ matrix 
  from the Gaussian unitary ensemble (GUE) with an external source which is specified by 
   the probability measure  on $\mathbb{R}^{n^2}$ with density 
\begin{equation}\label{meanGUE}
2^{\frac{1}{2}n(n-1)} \pi^{-\frac{1}{2}n^2}e^{-\mathrm{tr}(H-B)^2}, 
\end{equation} 
where $B$ is a deterministic $n\times n$ Hermitian matrix with eigenvalues denoted by $b_1, \ldots, b_n$,  we are devoted to studying  the eigenvalues of the Hermitised product matrix
\begin{equation}\label{W1}
W_M = G_M^* \cdots G_1^* H G_1 \cdots G_M
\end{equation}
under the assumption that all matrices, $H$ and $G_i$ ($i=1,\ldots,M$), are independent. We will see that the eigenvalues of $W_M$  form a determinantal point process as in the situation without $H$ which was first studied  by  Akemann et al. \cite{AIK13, AKW13}.

When $B=0$ in \eqref{meanGUE}, the global and local spectral properties   for the product \eqref{W1} have recently been studied in  \cite{FIL17}.  In particular, a new family of Meijer G-function type kernels
 at the origin was found therein, which is defined  for  $x, y\in  \mathbb{R} \setminus \{0\}$   by
 \begin{align} \mathcal{K}_{\nu_1,\ldots,\nu_M}^{\textup{(sub)}}(x,y)=  \frac{1}{2\pi i}\int_{C_{R}}dv\   G^{1,0}_{0,M+1} &\Big({\atop 0, -\nu_1,\ldots,-\nu_M}\Big|-\sgn(xy) |x|v\Big)
 \nonumber \\
&    \times  G^{M+1,0}_{0,M+1} \Big({\atop 0, \nu_1,\ldots,\nu_M}\Big| | y|v\Big)  \label{subcritical}\end{align}
with $C_{R}$ denoting a path in the right half-plane from $-i$ to $i$; see e.g. \cite{Ol10} for definition of Meijer G-functions. This is slightly different  from the Meijer $G$-kernel  defined for  $x,y>0$  by 
 \begin{align} K^{M}_{\textup{Meijer}}(x,y)=  \int_{0}^{1}du\,
  G^{1,0}_{0,M+1} \Big({\atop 0, -\nu_1,\ldots,-\nu_M}\Big| xu\Big) 
G^{M,0}_{0,M+1} \Big({\atop 0, \nu_1,\ldots,\nu_M}\Big| yu \Big),  \label{Meijer}\end{align}
which was first obtained  in \cite{KZ14} for the product of independent  Ginibre matrices, i.e.,   \eqref{W1} but with $H=I_n$.

Actually,  the past few years have witnessed a very rapid development in the  topic of products of  independent  random matrices. 
A crucial advance was the derivation of exact   eigenvalue density for the  product   \eqref{W1} with $H=I_n$  by
Akemann and his coworkers \cite{AKW13,AIK13}, which shows that it  forms a determinantal point process.  Subsequently, it was shown by Kuijlaars and Zhang in
 \cite{KZ14} that 
the corresponding correlation kernel  admits a double integral formula.  
These  have opened up the possibility to investigate  local statistical properties of eigenvalues.  Actually, a new family of limiting kernels, so-called   Meijer $G$-kernels \eqref{Meijer}, was found in \cite{KZ14}  at the
hard edge  and  the standard Sine and Airy kernels in the bulk and soft edge of the spectrum was proved in  \cite{LWZ14}.   Even more interestingly,  the   Meijer $G$-kernel also appears in other product ensembles \cite{Fo14,KS14,KKS15} and Cauchy matrix models \cite{BB15, BGS14}.  All these studies  form part of a fast paced and very recent literature 
relating to the  integrability and universality of random matrix products. We refer the reader to \cite{AI15} for a recent survey.

 In another special case when $M=0$,  \eqref{W1} reduces to the well-known Gaussian Unitary Ensemble with external source (also called deformed GUE ensemble in the literature).
 The  deformed GUE  ensemble  has  been treated in a series of papers \cite{AM07, ABK05, BK04, BK07, BH96, BH98, CP16, CKW15, CW14, Jo01, Pe06,  Sh11,TW06}.  More generally,  see \cite{KY14, LS16, LSSY} and references therein  for deformed  Wigner  ensembles.  As the eigenvalues of  the source $B$ change at a certain critical rate,  except that  there exists a  phase transition for largest eigenvalues due to Baik, Ben Arous and P\'{e}ch\'{e} \cite{BBP, Pe06} (sometimes  called BBP transition in the literature),  another  interesting phenomenon  will appear at the origin and can be  described by the so-called Pearcey kernel \cite{BH98} (a very special case of \eqref{critical}  below  where $M=0$ and $p=0$). See    also  \cite{AM07, BK07, CP16, OR07, TW06} for  the Pearcey kernel.

 It is worth stressing that the product \eqref{W1} with $H=(G_{0}+A)^{*}(G_{0}+A)$,  where $G_0$ is a Ginibre matrix and $A$ is a deterministic matrix, 
 has been studied in \cite{FL16}.  As the source matrix $A$ changes,   a  phase transition phenomenon for smallest singular values is observed  at the origin. 
 In particular, there exists a new family of kernels defined in terms of Meijer $G$-functions at  the critical value.  It's our goal  in the present  paper  to prove the existence of  a   phase transition  at the origin 
for  the product \eqref{W1} with $H$ distributed according to  the density \eqref{meanGUE}.

\subsection{Main results}
 Let  $\Delta_{n}(x)= \prod_{1\leq i<j\leq n}(x_j-x_i)$ denote the Vandermonde determinant.  
We are ready to state our main results as follows. 

The first is about the eigenvalue probability density function (PDF for short)  of the product \eqref{W1} and  can be derived after a direct application of    \cite[Lemma 2]{FIL17}. 

\begin{prop}\label{pdfGUEthm} 
Let  $\nu_0=0, \nu_1, \ldots, \nu_M$ be non-negative integers.  Suppose that 
$H$ is a random  $n\times n$ Hermitian matrix  with density  \eqref{meanGUE} and that  $G_1, \ldots, G_M$  are independent  standard complex Gaussian matrices where  $G_m$ is of size $(\nu_{m-1}+n) \times (\nu_{m} +n)$, independent of $H$.  Then  the  joint PDF for non-zero eigenvalues of  the product $W_{M}$  defined in ~\eqref{W1} is given by
\begin{equation}\label{PDF-meanGUE}
P_{n,M}(x_1,\ldots,x_n)=\frac{1}{Z_{n,M}} \Delta_{n}(x)\, \det\big[g_{M}( x_i, b_j)\big]_{i,j=1}^{ n},  \quad x_1,\ldots,x_n \in  \mathbb{R} \setminus \{0\},
\end{equation}
where   $g_{M}$ is a   function of two variables  defined  for  $(y, v)  \in  \mathbb{R}\setminus \{0\} \times  \mathbb{C}$ by   \begin{equation}\label{weight}
g_{M}(y,v)=\int_0^\infty \frac{dt_1}{t_1} \cdots  \int_0^\infty \frac{dt_M}{t_M} \,\prod_{l=1}^{M} t_{l}^{\nu_l}e^{-t_{l}}\,  \exp\!\Big\{-\frac{y^2}{(t_{1}\cdots t_{M})^2}+\frac{2yv}{t_{1}\cdots t_{M}}\Big\}, 
\end{equation}
and the normalisation constant 
\begin{equation}
Z_{n,M}=n!\, e^{\sum_{l=1}^{n}b_{l}^2} \Delta_{n}(b)\prod_{m=1}^M\prod_{j=1}^n\Gamma(\nu_m+j). \end{equation}
When some of $b_{j}$'s coincide,  L'Hospital's rule provides an appropriate   density.
\end{prop}

Our second result is a double integral representation of  correlation kernel  for  the  bi-orthogonal ensemble \eqref{PDF-meanGUE}  as a   determinantal point process (see e.g. \cite{Bo99} for  the  bi-orthogonal ensemble with more details).  For this, let us introduce one auxiliary  function, which is defined for  non-negative integers $\nu_1, \ldots, \nu_M$ and  for 
$(x, u)  \in  \mathbb{R}\times  \mathbb{C}$ by
  \begin{align}f_{M}(x,u)&= \frac{1}{(2\pi i)^{M}}     \int_{\mathcal{C}_{0}} \frac{d s_1}{s_1}\cdots \int_{\mathcal{C}_{0}}\frac{d s_M}{s_M} \,   \prod_{l=1}^M s_{l}^{-\nu_l}e^{s_l}  \,    \exp\Big\{\frac{x^2}{(s_1\cdots s_M)^2}-\frac{2xu}{s_1\cdots s_M}\Big\} 
 \label{function1}
\end{align}
where  $ \mathcal{C}_0 $  is an  anticlockwise loop around the origin. 

 Note that  when $M=0$ $f_{0}(x,u)=e^{x^2-2xu}$ and $g_{0}(y,v)=e^{-y^2+2yv}$, by convention. Moreover, it is easy to  verify two simple facts: (1)    $|f_{M}(x,u)| \leq C(x)  e^{|xu|}$ for some constant depending on $x$, just by letting  each contour be a unit circle  and noting the inequality $|e^{z}|\leq e^{|z|}$;  (2) $ g_{M}(y,v) $  is an analytic function of $v$   whenever $y \in  \mathbb{R} \setminus \{0\}$.

 \begin{thm}\label{kernelmeanGUE} With two functions  defined in \eqref{weight} and  \eqref{function1},  
the correlation  kernel  associated with the eigenvalue PDF \eqref{PDF-meanGUE} is given  by 
 \begin{equation}
  K_{n}(\textbf{b};x,y)= \frac{1}{2(\pi i)^2}  \int_{\mathcal{L}} du \int_{\mathcal{C}_{\textbf{b}}} dv f_{M}(x,u) g_{M}(y,v)\, e^{u^{2}-v^{2}} \frac{1}{u-v}\prod_{l=1}^n \frac{u-b_l}{v-b_l},  \label{integralkernel}
 \end{equation}
  where $\mathcal{C}_{\textbf{b}}$ encircles   $b_1, \ldots, b_n$  in an anticlockwise direction, and $\mathcal{L}$ is a path from $-i\infty$ to $i\infty$   not crossing  $\mathcal{C}_{\textbf{b}}$.
 \end{thm}

The third  is the key result of the present paper.  It describes a   phase transition phenomenon of eigenvalues  at the origin,   as the source matrix  $B$ changes. Specifically,  
 except for  finitely many eigenvalues of $B$, say $b_1, \ldots, b_r$,  we assume  that one half of  the rest  are equal to $\sqrt{n/2}a$  and the other  half  $-\sqrt{n/2}a$.   As  $n$ goes to infinity, we observe three different  families of limiting  kernels.
\begin{thm} [Phase transition at the origin]\label{transition}  With   the kernel   \eqref{integralkernel}, 
 let $r$ be a fixed nonnegative integer such that $n-r=2n_0$ is even, and suppose that 
 \begin{equation} 
 b_{r+1}=\cdots=b_{r+n_{0}}=-b_{r+n_{0}+1}=\cdots=-b_{n}=\sqrt{n/2}a, \quad a\geq 0. \label{finiterank}
 \end{equation}
The following hold true  uniformly for $x, y$ in a compact  set of  $\mathbb{R}\setminus\{0\}$. 

(i)  When  $0\leq a \leq \sqrt{2}/2$, let  $b_l=\sqrt{n/2}a_{l}$ such that $|a_l|<a+1$ for $l=1, \ldots, r$, then 
\begin{equation}
 \lim_{n\rightarrow \infty}\frac{1}{ \sqrt{2(1-a^{2})n}} K_{n}\Big(\textbf{b};\frac{x}{ \sqrt{2(1-a^{2})n}},\frac{y}{ \sqrt{2(1-a^{2})n}}\Big)=\mathcal{K}_{\nu_1,\ldots,\nu_M}^{\textup{(sub)}}(x,y),
\end{equation}
 where  $\mathcal{K}_{\nu_1,\ldots,\nu_M}^{\textup{(sub)}}(x,y)$ is defined by   \eqref{subcritical}.
%

 (ii)  When  $a=(1-\frac{\tau}{2\sqrt{n}})^{-1}$ with real $\tau$, for $0\leq p_{0}\leq p\leq r$ let 
\begin{equation} b_1= 2^{-\frac{1}{2}} n^{\frac{1}{4}}  a_1,  \ldots , b_p=2^{-\frac{1}{2}} n^{\frac{1}{4}} a_p,\   \ a_{1}\leq \cdots \leq a_{p_0}<0<a_{p_{0}+1}\leq \cdots \leq a_{p}, \end{equation} and  let $b_l=\sqrt{n/2}a_{l}$ with $a_{l}>0$ for $l=p+1,  \ldots, r$, then 
\begin{equation}
 \lim_{n\rightarrow \infty}\frac{1}{ \sqrt{2\sqrt{n}}} K_{n}\Big(\textbf{b};\frac{x}{  \sqrt{2\sqrt{n}}},\frac{y}{  \sqrt{2\sqrt{n}}}\Big)=\mathcal{ K}_{\nu_1,\ldots,\nu_M}^{\textup{crit}}(\tau;x,y)
\end{equation}
 where 
 \begin{align} \mathcal{K}_{\nu_1,\ldots,\nu_M}^{\textup{crit}}&(\tau;x,y)= \frac{1}{(2\pi i)^2}  \int_{i\mathbb{R}}  du \int_{ \Sigma_{-}\cup \Sigma_{+}} dv \,  e^{\frac{\tau}{2} ( u^{2}-v^{2})-\frac{1}{4}(u^{4}- v^4)} \frac{1}{u-v} \prod_{j=1}^{p}\frac{u-a_j}{v-a_j}  \nonumber \\ 
 &\quad \times G^{1,0}_{0,M+1} \Big({\atop 0, -\nu_1,\ldots,-\nu_M}\Big|xu\Big) G^{M+1,0}_{0,M+1} \Big({\atop 0, \nu_1,\ldots,\nu_M}\Big|-yv\Big). \label{critical}\end{align}
Here $\Sigma_{-}$ is a path in the left half-plane from $e^{-3i\pi/4}\infty$ to $e^{3i\pi/4}\infty$  with    $a_{1}, \ldots, a_{p_0}$ to its left side, while $\Sigma_{+}$ is a path in the right half-plane from $e^{i\pi/4}\infty$ to $e^{-i\pi/4}\infty$  with  $a_{p_{0}+1}, \ldots, a_{p}$ to its right side.

(iii)  When  $a>1$, for $1< p\leq r$
 let $b_l=a_{l}$  for $l=1,  \ldots, p$  and  let $b_m=\sqrt{n/2}a_{m}$ with $a_{m}\neq 0$ for  $m=p+1,  \ldots, r$,  then 
\begin{align}
 \lim_{n\rightarrow \infty} K_{n}\Big(\textbf{b};x, y\Big)&=\mathcal{ K}_{\nu_1,\ldots,\nu_M}^{\textup{sup}}(x,y),\end{align}
where 
\begin{equation}
 \mathcal{ K}_{\nu_1,\ldots,\nu_M}^{\textup{sup}}(x,y)= \frac{1}{2(\pi i)^2}  \int_{\mathcal{L}} du \int_{\mathcal{C}_{\textbf{a}}} dv f_{M}(x,u) g_{M}(y,v)\, e^{(1-\frac{1}{a^2})(u^{2}-v^{2})} \frac{1}{u-v}\prod_{l=1}^p \frac{u-a_l}{v-a_l}. \label{supkernel}
 \end{equation}
 \end{thm}

\begin{remark}
 We believe part (i) of Theorem \ref{transition}   holds  true whenever $a\in [0,1)$, as in the GUE ensemble  with external source; see e.g. \cite{ABK05}.  The  reason that we impose restrictions  on $a$  is mainly  because of the choice of contours. If we could remove the restriction stated in   Proposition \ref{contoursub} of Sect. \ref{sect:hardlimits}, then part (i) of Theorem \ref{transition}   holds true too.
  
  \end{remark}

The rest of the article is organised as follows. 
In Section~\ref{sect:PDF&kernel}, we  derive  the eigenvalue PDF for the product~\eqref{W1} as a bi-orthogonal ensemble and give an  explicit  double integral for correlation kernel.  
The scaling  limits are at the origin  are proved in Section~\ref{sect:hardlimits}.  Finally, in Section \ref{sect:density}  we give some discussion on   the global density.

\section{Eigenvalue PDF and double integral for correlation kernel} \label{sect:PDF&kernel}


 With Lemma 2 of \cite{FIL17}  at hand, we are immediately ready to write down the eigenvalue  PDF   for the product~\eqref{W1} and thus give a proof of Proposition \ref{pdfGUEthm}.

\begin{proof}[Proof of  Proposition \ref{pdfGUEthm}]  It is sufficient to derive the eigenvalue PDF of  the product  $H (G_1 \cdots G_M)(G_1 \cdots G_M)^{*}$. 
Since  the product $(G_1 \cdots G_M)(G_1 \cdots G_M)^{*}$ is only  involved, equivalently, we   can  suppose that each $G_j$ is an $n\times n$ random  matrix with density proportional to  ${\det\!{(G_{j}^* G_{j})}}^{\nu_j} \exp\{-\tr(G_{j}^* G_{j})\}$  according to   the results from   \cite{AIK13,IK14}.  It is well-known that the eigenvalue PDF of an $n\times n$ GUE matrix  with an external source is given by~\eqref{PDF-meanGUE} with $M=0$, i.e. $g_{0}(y, v)=\exp\{-y^2+2yv\}$ (see e.g.  \cite{Jo01} or \cite{Fo10}). Note that  when  $G$ is a square matrix and is distributed as ${\det\!{(G^* G)}}^{\nu} \exp\{-\tr(G^* G)\}$ up to a normalisation constant,  Theorem 1  of \cite{FIL17} holds true,  so does Lemma  2 of \cite{FIL17}. We thus complete the proof   after repeating the lemma $M$ times. 
\end{proof}


 Next,  we settle down to the derivation of double contour integrals for correlation kernel of  the bi-orthogonal ensemble \eqref{PDF-meanGUE}.

\begin{proof} [Proof of Theorem \ref{kernelmeanGUE} ]  First, we need to compute the  moment matrix $B_n=(b_{i,j})_{i,j=}^{n}$ via Hermite polynomials  and their integral representations given  by 
\be 
H_{m}(z):=(-1)^{m} e^{z^2} \frac{d^{m}z}{dz^{m}}e^{-z^2}=\frac{1}{\sqrt{\pi} }  \int_{-\infty}^{\infty}(2ix)^{m}  e^{-(x+iz)^2} dx, 
\ee
and get  
\begin{equation}b_{k,\ell}:=\int_{-\infty}^\infty x^{k-1}g_{M}(x, b_\ell)dx =\sqrt{\pi}(2i)^{-k+1} e^{b_{\ell}^{2}}H_{k-1}(ib_{\ell})  \prod_{m=1}^M\Gamma(\nu_{m}+k). \end{equation} Let $C_n=(c_{k,l})$ be the inverse of  $B_n$, then the correlation kernel for the bi-orthogonal ensemble  \eqref{PDF-meanGUE}  can be rewritten as a  summation 
\begin{equation}\label{meankernel}
K_n(\textbf{b};x,y)=\sum_{k,\ell=1}^{n} c_{\ell,k}\,x^{k-1} g_{M}(y, b_\ell), 
\end{equation}
see e.g. \cite[Proposition 2.2]{Bo99}.  

The entries $c_{i,j}$ of $C_n$  satisfy  the relation  $\sum_{k=1}^{n}   c_{j,k} b_{k,\ell} =\delta_{j,\ell}$,  which we specify for  
\begin{equation}   \sum_{k=1}^{n} \sqrt{\pi}(2i)^{-k+1} e^{b_{\ell}^{2}}H_{k-1}(ib_{\ell}) \prod_{m=1}^M\Gamma(\nu_{m}+k)  \   c_{j,k}=\delta_{j,\ell}, \quad j, \ell=1, \ldots, n. \label{psum}\end{equation}
Without loss of generality,  we assume that $b_{1}, \ldots, b_n$ are pairwise distinct. The above equations immediately imply
\begin{equation}    \sum_{k=1}^{n} \sqrt{\pi}(2i)^{-k+1} H_{k-1}(i u) \prod_{m=1}^M\Gamma(\nu_{m}+k)  \   c_{j,k}=e^{-b_{j}^{2}}\prod_{l=1,l\neq j}^{n}  \frac{u-b_{l}}{b_{j}-b_{l}}   \label{sumeq}.\end{equation}
These  can be verified by noting that both sides  are polynomials of degree $n-1$ in $u$ and take the same values at $n$ different points  since    \eqref{psum}  holds true.

Using these implicit formulas for $\{c_{j,k}\}$ we are ready to   show that  \eqref{meankernel}  implies the double contour integral
formula \eqref{integralkernel}.  Use the integral representations  
\begin{equation} \frac{1}{\Gamma(l)}= \frac{1}{ 2\pi i }\int_{\mathcal{C}_{0}} s^{-l} e^{s}ds,  \qquad  (2iz)^{k-1}=\frac{1}{\sqrt{\pi} i}  \int_{\mathcal{L}}  e^{(u-z)^2}H_{k-1}(iu)du,  \end{equation}
where $\mathcal{L}$ is a path from $-i\infty$ to $i\infty$,  combine   the identity   \eqref{sumeq} and  we  rewrite
 \begin{align} \sum_{k=1}^{n} x^{k-1} c_{k,\ell}  &=\int_{\mathcal{C}_{0}}\frac{d s_M}{2\pi i} \cdots  \int_{\mathcal{C}_{0}}\frac{d s_M}{2\pi i}
   \prod_{l=1}^M  s_{l}^{-\nu_l-1}e^{s_l} \sum_{k=1}^n \,   \Big(\frac{2ix}{s_1\cdots s_M}\Big)^{k-1}   \nonumber\\ &\quad  \times (2i)^{-k+1} \prod_{m=1}^M\Gamma(\nu_{m}+k)  \,   c_{j,k} \nonumber\\
 &= \frac{1}{\pi i}  \int_{\mathcal{L}}   f_M(x,u)   e^{u^2-b_{\ell}^2}
     \prod_{l=1,l\neq \ell}^{n}  \frac{u-b_{l}}{b_{\ell}-b_{l}}   
   \end{align}
where we  have exchanged the order of integration and used  the definition  of  $f_M(x,u) $ \eqref{function1}.

  Finally,  recognising the  summation in \eqref{meankernel} over $\ell$ as the  summation  of the residues at $b_{1}, b_{2}, \ldots, b_n$  for the $v$-function 
 \be g_M(y,v) e^{-v^2}
  \frac{1}{u-v} \prod_{l=1}^{n}  \frac{u-b_{l}}{v-b_{l}}  \ee
  and using Cauchy residue theorem, we  thus arrive at   the formula  \eqref{integralkernel} by  choosing  two disjoint contours   $\mathcal{L}$ and $\mathcal{C}_{\textbf{b}}$. 
\end{proof}

\section{Scaling limits at the origin} \label{sect:hardlimits}

In this section we prove part (i), (ii) and (iii)  of Theorem \ref{transition}  in turn. 

\begin{proof}[Proof of Theorem \ref{transition}: part  (i)]
By the assumptions on $b_1, \ldots, b_n$, substitute $u, v$  by  $\sqrt{n/2}u$, $\sqrt{n/2}v$ respectively in \eqref{integralkernel} and  we obtain  for $\xi=x/\sqrt{1-a^2}$ and $\eta=y/\sqrt{1-a^2}$
  \begin{multline}   \frac{1}{\sqrt{2n}}K_n\Big(\textbf{b};\frac{\xi}{\sqrt{2n}}, \frac{\eta}{\sqrt{2n}}\Big)= \frac{1}{(2\pi i)^2}\int_{\mathcal{L}} du \int_{\mathcal{C}} d v  
 f_{M}(\xi/\sqrt{2n},\sqrt{n/2}u)  g_{M}(\eta/\sqrt{2n},\sqrt{n/2}v) 
 \\ \times   \,  e^{n(h(u)-h(v))} \frac{1}{u-v} \Big(\frac{u^{2}-a^{2}}{v^{2}-a^2} \Big)^{-r/2}\prod_{l=1}^r \frac{u-a_l}{v-a_l}, \label{rescalingkernel}
     \end{multline}
     where    $\mathcal{C}$ encircles   $\pm a, a_1, \ldots, a_r$ and $\mathcal{L}$ is a path from $-i\infty$ to $i\infty$, and  the phase function 
     \be h(z)=\frac{1}{2}z^2 +\frac{1}{2}\log(a^{2}-z^{2}). \ee
     
     Since  
     
     \be h'(z)=z-\frac{z}{a^{2}-z^{2}},\ee
     we easily know that the equation $h'(z)=0$ has  three solutions 
     \be z_0=0, \qquad z_{\pm}=\pm i\sqrt{1-a^{2}}, \ee
     from which we distinguish  three scenarios:  (i)  $0\leq a<1$;   (ii)  $a=1$;  (iii)  $a>1$. When $a=1$,  the three simple  saddle points  coalesce into  a third-order point  at zero and thus this is a critical case. 
          
Although both the functions $ f_M$ and $g_M$ in the integrand  of \eqref{rescalingkernel} depend on $n$, we will see below  that for the large $n$ they do not enter the saddle point equation.   So we may perform saddle-point approximations and this is what we will do next in details.

    In order to investigate the case (i) with $0\leq a<1$,  we  first  proceed  to consider the situation  $\eta<0$. For this,  we need to deform the integral contours as follows. Given $\delta\geq 0$, let's first define  $\mathcal{C}_{R,\delta}$ as a great arc  along the circle with radius $\sqrt{1+\delta^{2}+2a\delta}$ and centre at  $a+\delta$, which is entirely   in the right-half plane and connects the two points  $-i\sqrt{1-a^2}$  and  $i\sqrt{1-a^2}$.   Let  $\mathcal{C}_{L,\delta}$ be the reflection of $\mathcal{C}_{R,\delta}$ about the $y$-axis.  Choose $\mathcal{C}=\mathcal{C}_{L,0}\cup \mathcal{C}_{R,0}$ and deform $\mathcal{L}$ as the union of the $y$-axis and $ \tilde{\mathcal{C}}_{R}:=\mathcal{C}_{R,0.1}\cup\{(0,y): |y|\leq \sqrt{1-a^{2}}\}$, with $\tilde{\mathcal{C}}_{R}$ in a  counterclockwise direction and the $y$-axis  from $-i\infty$ to $i\infty$. 
    Note the assumption on $a_1, \ldots, a_r$,  such a choice assures that $ \tilde{\mathcal{C}}_{R}$ encircles $\pm a, a_1, \ldots, a_r$. Divide the integration over  $\mathcal{L}$  into two parts,  we further rewrite  the double   integral on the RHS of  \eqref{rescalingkernel} as a sum of two integrals \begin{align}  \frac{1}{\sqrt{2n}}K_n\Big(\textbf{b};\frac{\xi}{\sqrt{2n}}, \frac{\eta}{\sqrt{2n}}\Big)&=
 \textrm{P.V.}\int_{i\mathbb{R}} du \int_{\mathcal{C}} dv \Big(\cdot\Big)+ \int_{ \tilde{\mathcal{C}}_{R}}du \int_{\mathcal{C}} dv  \Big(\cdot\Big) :=I_{1}+I_{2}.
   \label{twointegrals}\end{align}
  Here the notation $\textrm{P.V.}$ denotes the Cauchy principal value  integral.  
  
     As   $n\rightarrow \infty$,  we claim that   the integral over the range of $v\in \mathcal{C}_{R,0}$  and $u\in \tilde{\mathcal{C}}_{R}$ gives rise to a leading  contribution  to  the double   integral on the RHS of \eqref{rescalingkernel}. Actually, for $I_2$, when $v\in \mathcal{C}_{L,0}$  the $u$-integral vanishes by Cauchy's theorem since the integrand does not have any singularity inside $\tilde{\mathcal{C}}_{R}$, while for $v\in \mathcal{C}_{R,0}$   application of the residue theorem shows 
   \begin{align}   &I_2= \frac{1}{2\pi i} \int_{\mathcal{C}_{R,0}} d v  
 f_{M}(\xi/\sqrt{2n},\sqrt{n/2}v) g_{M}(\eta/\sqrt{2n},\sqrt{n/2}v). \label{I2form}
  \end{align}

Consideration of the definition \eqref{function1} permits us to  get as $n \rightarrow \infty$ 
  \be    f_{M}(\xi/\sqrt{2n},\sqrt{n/2}v)  \sim  \sum_{k=0}^{\infty} \frac{(-\xi v)^k}{k!}\prod_{l=1}^{M}\frac{1}{\Gamma(\nu_l+1+k)}, \label{1fasymtotics}\ee
 the RHS of which is recognized as a Meijer G-function  via  \be  \sum_{k=0}^{\infty} \frac{(-z)^k}{k!}\prod_{l=1}^{M}\frac{1}{\Gamma(\nu_l+1+k)}=G^{1,0}_{0,M+1} \Big ({ \underline{\hspace{0.5cm}}
 \atop 0,-\nu_1, \dots,-\nu_M} \Big |  z \Big ). \label{hyper-Grelation}\ee Here  the notation $f_{1,n} \sim f_{2,n}$  means that $\lim_{n \to \infty} f_{1,n}/f_{2,n} = 1$. 
 
 However,  in order to obtain the leading asymptotic behaviour of  $g_{M}(\eta/\sqrt{2n},\sqrt{n/2}v)$,  we need to derive   a Mellin-type integral representation of $g_{M}(y,v)$ for $(y, v)  \in  \mathbb{R}\setminus \{0\} \times  \mathbb{C}$
 \begin{align} g_{M}(y,v) 
 =\frac{e^{v^{2}/2}}{2\pi i} \int_{c-i\infty}^{c+i\infty} ds\, \big(\sqrt{2}|y|\big)^{-s}   U\big(s-\frac{1}{2},-\sqrt{2}\sgn(y) v\big)\, \prod_{l=0}^{M} \Gamma(\nu_l+s), 
  \label{function2}\
  \end{align}
  where $c>0$ and the parabolic cylinder function 
 \begin{equation} U(c,z)= \frac{e^{-\frac{1}{4}z^2}}{\Gamma(c+\frac{1}{2})}\int_{0}^{\infty}  t^{c-\frac{1}{2}}e^{-\frac{1}{2}t^{2}-zt} dt, \qquad \textup{Re}(c)> -\frac{1}{2}. \label{Ufun} \end{equation}
 This can   be proved from  \eqref{weight}  by  applying Mellin  and inverse  Mellin transforms  if  $ g_{M}(y,v)$  is treated as a function of the variable $y$ over $(0, \infty)$. When   $y \in(-\infty, 0)$, we just turn to  consider the variable $-y>0$.  
 
  Using  asymptotic expansion  of the parabolic cylinder function \eqref{Ufun} as $z\to \infty$ (see e.g. \cite[Sect. 12.9]{Ol10})
   \be  U(c,z) =\begin{cases} z^{-c-\frac{1}{2}}e^{-\frac{1}{4}z^2}\Big(1+\mathcal{O}(\frac{1}{z^2})\Big), & |\textup{ph}(z)|<\frac{3}{4}\pi, \\
   \frac{1}{\Gamma(c+\frac{1}{2})} \sqrt{2\pi} (-z)^{c-\frac{1}{2}}e^{\frac{1}{4}z^2}\Big(1+\mathcal{O}(\frac{1}{z^2})\Big), & \frac{3}{4}\pi<\textup{ph}(z)<\frac{5}{4}\pi, \label{asyU}\end{cases}\ee
 we have 
\be    g_{M}(\eta/\sqrt{2n},\sqrt{n/2}v)  \sim   \frac{1}{2\pi i}\int_{c-i\infty}^{c+i\infty} ds\, \big(|\eta|v\big)^{-s}   \, \prod_{l=0}^{M} \Gamma(\nu_l+s)=G^{M+1,0}_{0,M+1} \Big ({ \underline{\hspace{0.5cm}}
 \atop 0, \nu_1, \dots,\nu_M} \Big |  |\eta| v \Big ). \label{2fasymtotics}\ee 
  Combining \eqref{I2form}, \eqref{1fasymtotics} and \eqref{2fasymtotics},  after a change of variables we see that  $I_2$ leads to the limiting kernel in part (ii).

 Next,  we   deal with the integral $I_1$ and show that it is negligible compared to $I_2$. In this case because of different asymptotic forms of $g_M$, we divide $I_1$ into two parts again as
 \begin{align}   I_1= \textrm{P.V.}\int_{i\mathbb{R}} du \int_{\mathcal{C}_{1,+}} dv \Big(\cdot\Big)   +\int_{i\mathbb{R}} du \int_{\mathcal{C}_{1,-}} dv \Big(\cdot\Big):=I_{11}+I_{12},
   \label{twosubintegrals}\end{align}
 where $\mathcal{C}_{1,+}= \{z\in \mathcal{C}_{L,0}\cup \mathcal{C}_{R,0}: |\textup{ph}(z)|<\frac{3}{4}\pi \}$ and 
 $\mathcal{C}_{1,-}= \{z\in \mathcal{C}_{L,0}: \frac{3}{4}\pi<\textup{ph}(z)<\frac{5}{4}\pi\}$.

 When  $0\leq a\leq \sqrt{2}/2$, Proposition \ref{contoursub} below shows that $\textrm{Re}\{h(z)\}$   attains its global   minimum at $\pm i\sqrt{1-a^2}$ over  $\mathcal{C}_{L,0}\cup \mathcal{C}_{R,0}$,  and     attains its global maximum at $\pm i\sqrt{1-a^2}$ over $i\mathbb{R}$.   Therefore,  for $I_{11}$, combining  \eqref{function1}, \eqref{function2} and \eqref{asyU} we obtain
    \begin{multline}    I_{11}\sim \textrm{P.V.}   \frac{1}{(2\pi i)^2} \int_{i\mathbb{R}} du \int_{\mathcal{C}_{1,+}} dv  
 e^{n(h(u)-h(v))} \frac{1}{u-v} \Big(\frac{u^{2}-a^{2}}{v^{2}-a} \Big)^{-r/2}\prod_{l=1}^r \frac{u-a_l}{v-a_l} \\ \times
   G^{1,0}_{0,M+1} \Big ({ \underline{\hspace{0.5cm}}
 \atop 0,-\nu_1, \dots,-\nu_M} \Big |  \xi u \Big ) G^{M+1,0}_{0,M+1} \Big ({ \underline{\hspace{0.5cm}}
 \atop 0,\nu_1, \dots, \nu_M} \Big |  |\eta| v  \Big ).
 \end{multline}
   For this,   the standard steepest  descent argument  shows that  the   leading  term  for the integral $I_{11}$  comes from the neighbourhood of the saddle points $(z_{+}, z_{+})$ and $(z_{-}, z_{-})$ and can be estimated by  
    \be I_{11}=\mathcal{O}(\frac{1}{\sqrt{n}}). \label{I11term}\ee
    
     Similarly, for  $I_{12}$,  combination of \eqref{function1}, \eqref{function2} and \eqref{asyU}  gives  rise to 
    \begin{align}    &I_{12}\sim      \frac{1}{(2\pi i)^2} \int_{i\mathbb{R}} du \int_{\mathcal{C}_{1,-}} dv 
  \Big(\frac{u^{2}-a^{2}}{v^{2}-a} \Big)^{-r/2}\prod_{l=1}^r \frac{u-a_l}{v-a_l} G^{1,0}_{0,M+1} \Big ({ \underline{\hspace{0.5cm}}
 \atop 0,-\nu_1, \dots,-\nu_M} \Big |  \xi u \Big )\nonumber  \\ &\times
   \frac{e^{n(h(u)-h(i\sqrt{1-a^2}))}}{u-v}  \int_{c-i\infty}^{c+i\infty} \frac{ds}{\sqrt{2\pi} i} \, |\eta|^{-s} (-nv)^{s-1} \sqrt{n}  e^{-\frac{n}{2}(\log(a^2-v^{2})+1-a^{2})}  \prod_{l=1}^{M} \Gamma(\nu_l+s).
 \end{align}
  We claim that the  integrals of $u$ and $v$  respectively afford us  bounds $\mathcal{O}(n^{-\frac{1}{2}})$ and  $\mathcal{O}(n^{c-\frac{1}{2}}e^{-\frac{1}{2}(1-a^2)n})$.  The former  can be obtained via the steepest  descent argument. For the latter,  writing  $v=-a+e^{i\theta}\in \mathcal{C}_{L,0}$,  it is seen  from $\cos\theta \leq a$ that 
  \begin{align}  \textrm{Re}\{\log(a^2-v^{2})\}+1-a^{2}=\frac{1}{2}\log(1+4a^{2}-4a\cos\theta) +1-a^2 \geq 1-a^2. \end{align} Together,  we arrive at an exponential decay estimation
    \be I_{12}=\mathcal{O}(n^{c-1} e^{-n(1-a^2)/2})). \label{I12term}\ee

    Combining  \eqref{I2form}, \eqref{I11term} and \eqref{I12term},     note   \eqref{twointegrals}   and \eqref{twosubintegrals} and  we  complete  the proof  of  part (i) for $\eta<0$.  
  
    The proof in the case of $\eta>0$ is very similar.  But this time we need to deform $\mathcal{L}$ as the union of the $y$-axis and $ \tilde{\mathcal{C}}_{L}:=\mathcal{C}_{L,0.1}\cup\{(0,y): |y|\leq \sqrt{1-a^{2}}\}$ with $\tilde{\mathcal{C}}_{L}$ being  clockwise.   
 
    Finally, it is easily seen that  the previously derived  estimates are valid  uniformly for $\xi, \eta$ in any given compact set of  $\mathbb{R}\setminus\{0\}$.    
\end{proof}

\begin{proof}[Proof of Theorem \ref{transition}: part  (ii)]   
Substituting $u, v$  by  $\sqrt{n/2}u, \sqrt{n/2}v$  in \eqref{integralkernel}, by the assumptions   we obtain  
  \begin{multline}   \frac{1}{\sqrt{2\sqrt{n}}}K_n\Big(\textbf{b};\frac{x}{\sqrt{2\sqrt{n}}}, \frac{y}{\sqrt{2\sqrt{n}}}\Big)= \frac{n^{\frac{1}{4}}}{(2\pi i)^2}\int_{\mathcal{L}} du \int_{\mathcal{C}} d v\, e^{n(h(u)-h(v))} f_{M}\big(\frac{x}{\sqrt{2\sqrt{n}}},\sqrt{\frac{n}{2}}u\big)   
 \\ \times   \,  g_{M}\big(\frac{y}{\sqrt{2\sqrt{n}}},\sqrt{\frac{n}{2}}v\big)   \frac{1}{u-v} \Big(\frac{u^{2}-a^{2}}{v^{2}-a^2} \Big)^{-\frac{r}{2}}\prod_{l=1}^p \frac{n^{\frac{1}{4}}u-a_l}{n^{\frac{1}{4}}v-a_l} \prod_{m=p+1}^r \frac{u-a_m}{v-a_m}, \label{rescalingkernel-2}
     \end{multline}
     where     the phase function 
     \be h(z)=\frac{1}{2}z^2 +\frac{1}{2}\log(a^{2}-z^{2}),  \qquad a=(1-\frac{\tau}{2\sqrt{n}})^{-1}. \ee
     Here if $a$ is equal to the critical value 1, then the  three simple  saddle points  coalesce into  a third-order point    $z_0=0$.

     To  use the steepest  descent method to investigate   asymptotic behaviour of large  $n$,    we  need to choose proper contours according to Propositions  \ref{contoursub} and \ref{contourcrti} below.   For convenience, let's introduce some notations:  1)   $\delta$ is a fixed small  positive number; 2)  $\theta_0$ is  a bit larger than $\pi/4$, say $\theta_0=\frac{101}{400}\pi$,  so that it  satisfies   the condition of part (ii)  in Proposition \ref{contourcrti}; 3)  $q:=1+\max\{2, |a_1|, |a_{2}|, \ldots, |a_r|\}$; 4) $\mathcal{L}_{(x_1,y_1)\rightarrow (x_2,y_2)\rightarrow \cdots}$  denotes  the  union of line segments from points  $(x_1,y_1)$ to $(x_2,y_2)$ to $\cdots$.    Define $\mathcal{C}_{+}=\Gamma_{+,1}\cup  \Gamma_{+,2}$ with an anticlockwise direction 
     where 
     \be   \Gamma_{+,1}= \mathcal{L}_{(\delta \cos\theta_0, \delta \sin\theta_0) \rightarrow (0,0) \rightarrow (\delta \cos\theta_0, -\delta \sin\theta_0)},\ee 
     and 
       \begin{align}   \Gamma_{+,2}&=  \mathcal{L}_{(\delta \cos\theta_0, -\delta \sin\theta_0) \rightarrow (1,-\tan\theta_0)  \rightarrow  (q,-\tan\theta_0) \rightarrow (q,\tan\theta_0)\rightarrow (1,\tan\theta_0)  \rightarrow (\delta \cos\theta_0, -\delta \sin\theta_0)}.\end{align}
 Let  $\mathcal{C}_{-}=\Gamma_{-,1}\cup  \Gamma_{-,2}$ be the reflection  of $\mathcal{C}_{+}$  about  the $y$-axis  with  an anticlockwise direction, and let   $\mathcal{L}$ be the  $y$-axis.   We stress that it might be better to  deform  a small portion of $\Gamma_{+,1}$  near the origin    to the right a little such that it doesn't intersect  the $y$-axis,  however, our choice above  works  well because   they   just touch each other at the ``point of  tangency''.

  First, we divide the integral on the RHS of \eqref{rescalingkernel-2} into two parts
  \be    \textup{LHS \,of\,} \eqref{rescalingkernel-2} 
  = \int_{\mathcal{L}} du \int_{\Gamma_{+,1}\cup  \Gamma_{-,1}} d v \,\Big(\cdot\Big) + \int_{\mathcal{L}} du \int_{\Gamma_{+,2}\cup  \Gamma_{-,2}} d v \,\Big(\cdot\Big) :=I_{1}+I_{2}.\label{twopartssum-2}
  \ee  
     We claim that  the dominant contribution    comes   from the neighbourhood of   $(0,0)$, so we need to expand the function $h(z)$ at zero. With the double scaling  in mind, we obtain the Taylor series  \begin{align} h(z)
    =\log a+\frac{1}{2}\Big( \frac{\tau }{\sqrt{n}}- \frac{\tau^2 }{4n}\Big)z^2-\frac{1}{4} \Big(1-\frac{\tau}{2\sqrt{n}}\Big)^{4}z^{4}-\frac{1}{6} \Big(1-\frac{\tau}{2\sqrt{n}}\Big)^{6}z^{6}+\cdots, \label{Taylorexpansion}\end{align}
 from which   combining  \eqref{function1},  \eqref{function2} and \eqref{asyU}, together with the relation \eqref{hyper-Grelation}  and the definition of Meijer G-function,  we see that 
  \begin{multline}    I_{1}\sim   
  \frac{n^{\frac{1}{4}}}{(2\pi i)^2}\int_{\mathcal{L}} du \int_{\Gamma_{+,1}\cup  \Gamma_{-,1}} d v\, e^{n(h(u)-h(v))} \ G^{1,0}_{0,M+1} \Big ({ \underline{\hspace{0.5cm}}
 \atop 0,-\nu_1, \dots,-\nu_M} \Big |  n^{\frac{1}{4}}x u \Big )  \\ \times   G^{M+1,0}_{0,M+1} \Big ({ \underline{\hspace{0.5cm}}
 \atop 0,\nu_1, \dots, \nu_M} \Big |  -n^{\frac{1}{4}}yv  \Big )  \frac{1}{u-v} \Big(\frac{u^{2}-a^{2}}{v^{2}-a^2} \Big)^{-\frac{r}{2}}\prod_{l=1}^p \frac{n^{\frac{1}{4}}u-a_l}{n^{\frac{1}{4}}v-a_l} \prod_{l=p+1}^r \frac{u-a_l}{v-a_l}. \label{eqn2-1}
 \end{multline}
 Rescaling $u,v$ by    $n^{-\frac{1}{4}}$,  use   \eqref{Taylorexpansion} and we conclude that  the limit of $I_1$ leads to   the kernel \eqref{critical}, uniformly for  $x, y$ in a compact set  of   $\mathbb{R}\setminus\{0\}$. 

 Secondly,  for the integral $I_2$,   we divide it into two parts again
  \be    I_2= \int_{\mathcal{L}} du \int_{\Gamma_{+,2}} d v \,\Big(\cdot\Big) + \int_{\mathcal{L}} du \int_{\Gamma_{+,2}} d v \,\Big(\cdot\Big):=I_{2,+}+I_{2,-}, \label{twopartssum-22}
  \ee 
 and  show that they are ignorable  compared to $I_1$.   We  just  focus on the integral  $I_{2,-}$  since both are similar.  Because  of different asymptotic behaviour of  \eqref{asyU}, write 
  \be    I_{2,-}=  \int_{\mathcal{L}} du \int_{\Gamma_{-,2}^{(1)}} d v \,\Big(\cdot\Big) + \int_{\mathcal{L}} du \int_{\Gamma_{-,2}^{(2)}} d v 
  \,\Big(\cdot\Big):=I^{(1)}_{2,-}+I^{(2)}_{2,-}, \label{twopartssum-222}
  \ee 
 where $\Gamma_{-,2}^{(1)}=\{z\in \Gamma_{-,2}: |\textup{ph}(z)|<\frac{3}{4}\pi\}$ and $\Gamma_{-,2}^{(2)}=\{z\in \Gamma_{-,2}: \frac{3}{4}\pi<\textup{ph}(z)<\frac{5}{4}\pi\}$.
 
 Application of  \eqref{asyU} gives us the same asymptotic  form as in the RHS of \eqref{eqn2-1}  but with the $v$-contour  $\Gamma_{-,2}^{(1)}$, from which use of the steepest descent argument leads to an exponential decay.   However,    application of  \eqref{asyU}  to  $I^{(2)}_{2,-}$  yields
  \begin{multline}    I^{(2)}_{2,-} \sim  
  \frac{1}{(2\pi i)^3}\int_{\mathcal{L}} du \int_{\Gamma_{2,-}^{(2)}} d v\, e^{n(h(u)-h(0))} \ G^{1,0}_{0,M+1} \Big ({ \underline{\hspace{0.5cm}}
 \atop 0,-\nu_1, \dots,-\nu_M} \Big |  n^{\frac{1}{4}}x u \Big )  \\ \times    \frac{1}{u-v} \Big(\frac{u^{2}-a^{2}}{v^{2}-a^2} \Big)^{-\frac{r}{2}}\prod_{l=1}^p \frac{n^{\frac{1}{4}}u-a_l}{n^{\frac{1}{4}}v-a_l} \prod_{l=p+1}^r \frac{u-a_l}{v-a_l} 
   \\    \times 
 \,e^{-\frac{n}{2}\log(1-\frac{v^2}{a^2})}  \int_{c-i\infty}^{c+i\infty} ds\,  |y|^{-s} (\sgn(y)v)^{s-1} n^{(3s-1)/4}   \prod_{l=1}^{M} \Gamma(\nu_l+s).
 \end{multline}
 Note that   the endpoints of  $\Gamma_{2,-}^{(2)}$  are $(-\tan\theta_0, \pm \tan\theta_0)$,  with $a=(1-\frac{\tau}{2\sqrt{n}})^{-1}$ in mind, for sufficiently large $n$ we see  that       \begin{align}  \textrm{Re}\Big\{\log(1-\frac{v^2}{a^2})\Big\}& \geq-2\log a+ \frac{1}{2}\log\big((a+\tan\theta_0)^2+\tan^{2}\theta_0)(a-\tan\theta_0)^2+\tan^{2}\theta_0) \big) \nonumber\\ &\geq-2\log a+ 2 \log \tan\theta_0>\log \tan\theta_0>0 \end{align} 
holds true uniformly  for $\tau$ in a compact set of $\mathbb{R}$ and for for every $v\in \Gamma_{2,-}^{(2)}$,  use of the steepest descent argument  leads to an exponential decay \be  I^{(2)}_{2,-}=\mathcal{O}\big( n^{ \frac{3c-2}{4}}e^{-\frac{n}{2}\log \tan\theta_0 }\big).\ee

Lastly, by combining the foregoing results for $I_1$ and $I_2$, we  complete the proof of part (ii).
\end{proof}

\begin{proof}[Proof of Theorem \ref{transition}: part  (iii)]    Under  the assumptions   we can  rewrite \eqref{integralkernel} as 
  \begin{multline}  K_n(\textbf{b};x, y)= \frac{1}{2(\pi i)^2}\int_{\mathcal{L}} du \int_{\mathcal{C}} d v\, e^{u^{2}-v^{2}} f_{M}(x,u)  g_{M}(y,v) 
 \\ \times   \,    \frac{1}{u-v}  
 \Big(\frac{1-\frac{2u^2}{n a^{2}}
 }  {      1-\frac{2v^2}{n a^2}      } 
 \Big)^{\frac{n-r}{2}} \prod_{l=1}^p  \frac{u-a_l}{v-a_l} \prod_{m=p+1}^r \frac{u-\sqrt{\frac{n}{2}}a_m}{v-\sqrt{\frac{n}{2}}a_m}.  \label{rescalingkernel-31}
     \end{multline}
    Without loss of generality, we assume that $a_{p+1}, \ldots, a_r>0$.  Choose a fixed number $q$  such that $q>\max\{|a_1|,  \ldots, |a_p|\}$, and let 
    \be \mathcal{C}_{+}=\{z=q+t e^{\pm i \frac{\pi}{16}}:t\geq 0\} \quad \mathrm{and} \quad \mathcal{C}_{-}=\{z=-q+t e^{i\pi (1\pm \frac{15}{16})}: t\geq 0\} \ee
    both with an anticlockwise direction.  Then $\mathcal{C}_{-}$ encircles $ \sqrt{n/2}a$ and  $\mathcal{C}_{+}$ encircles $\sqrt{{n/2}}a$, $\sqrt{n/2}a_{p+1}, \ldots, \sqrt{n/2}a_{r}$, both not crossing  $\mathcal{C}_{\textbf{a}}$. For large $n$,  we choose $\mathcal{C}=\mathcal{C}_{-} \cup \mathcal{C}_{\textbf{a}} \cup  \mathcal{C}_{+}$
and  divide the integral on the RHS of \eqref{rescalingkernel-31} into three parts according to the $v$-contour, denoted by $I_{-},  I_{\textbf{a}},  I_{+}$. 
  
       Note that $a_{m}\neq 0$ for  $m=p+1,  \ldots, r$,  we easily see that the limit of  $I_{\textbf{a}}$ leads to the dominant contribution, while  
       \begin{equation}  I_{\pm} \rightarrow \frac{1}{2(\pi i)^2}\int_{\mathcal{L}} du \int_{\mathcal{C}_{\pm}} d v\, e^{(1-\frac{1}{a^2})(u^{2}-v^{2})} f_{M}(x,u)  g_{M}(y,v)    \,    \frac{1}{u-v}  
\prod_{l=1}^p  \frac{u-a_l}{v-a_l}=0 \label{rescalingkernel-32}
     \end{equation}
     since the integrand has no pole in $\mathcal{C}_{\pm}$  for the $v$-integral.    
       
       This completes the proof of part (iii).
\end{proof}

The following  two propositions are of importance  in choosing appropriate contours of integration  for   the method of steepest descent.

\begin{prop}  \label{contoursub}
 Let   $\mathcal{C}_{R}=\{z=a+e^{i\theta}: -\frac{\pi}{2}-\arccos a \leq \theta\leq \frac{\pi}{2}+\arccos a\}$ and let $\mathcal{C}_{L}$ be the reflection of $\mathcal{C}_{R}$ about the $y$-axis.  Then for $h(z)=\frac{1}{2}z^2 +\frac{1}{2}\log(a^{2}-z^{2})$, the following hold true. 
 
(i) When  $0\leq a\leq \sqrt{2}/2$,  $\textup{Re}\{h(z)\}$   attains its global   minimum at $\pm i\sqrt{1-a^2}$ over  $\mathcal{C}_{L}\cup \mathcal{C}_{R}$.
 
 (ii) When  $0\leq a\leq 1$,  $\textup{Re}\{h(z)\}$    attains its global maximum at $\pm i\sqrt{1-a^2}$ over $i\mathbb{R}$.  
\end{prop}

\begin{proof}
For  (i),  we first consider   $z\in \mathcal{C}_{R}$ and let $z=a+e^{i\theta}$. It is easy to obtain 
\be \textup{Re}\{h\}=\frac{1}{2}(2t^{2}+2at+a^2 -1)+\frac{1}{4}\log(4at+1+4a^{2}), \qquad t=\cos\theta.\ee
For this, we know from $t\in [-a,1]$  with $0\leq a\leq \sqrt{2}/2$ that \be \frac{d}{dt} \textup{Re}\{h\} =\frac{2}{4at+1+4a^{2}}(t+a)(4at+2a^{2}+1))\geq 0,\ee
and thus prove (i). The proof in the case  $z\in \mathcal{C}_{L}$ is similar.  

For (ii), let $z=iy$, we have 
\be \textup{Re}\{h\}=-\frac{1}{2} y^{2}+\frac{1}{2}\log(a^{2}+y^2), \qquad -\infty<y<\infty.\ee
Since \be \frac{d}{dt} \textup{Re}\{h\} =\frac{y}{a^{2}+y^2}(1-a^{2}-y^{2}),\ee
the maximum can be obtained at $y=\pm \sqrt{1-a^2}$.
 \end{proof}
 
 \begin{prop}  \label{contourcrti}
Let   $h(z)=\frac{1}{2}z^2 +\frac{1}{2}\log(1-z^{2})$, for $\theta_0 \in \mathbb{R}$  the following hold true. 

(i)  When  $ \frac{\pi}{8} \leq |\theta_0| \leq \frac{\pi}{4}$,    $\textup{Re}\{h(te^{i\theta_0})\}$ is a strictly increasing  function of $t$ over  $[0, \infty)$.

(ii)  When  $ \frac{\pi}{4} < |\theta_0| \leq \frac{\pi}{2}-\frac{1}{2}\arccos\frac{2-\sqrt{2}}{2}$,    $\textup{Re}\{h(te^{i\theta_0})\}$ is a strictly increasing  function of $t$ over $[0, t_{max})$ with $t_{max}=\sqrt{2\cos 2\theta_0 -1/\cos 2\theta_0}$. Moreover,   $t_{max} \cos\theta_0\geq 1$.

(iii)  For $0\neq y\in \mathbb{R}$, $\textup{Re}\{h(x+iy)\}$   is a strictly increasing  function of $x$ over $[1,\infty)$. 

(iv) When $x\geq 2$, $\textup{Re}\{h(x+iy)\}$   is a strictly decreasing (resp. inceasing) function of $y$ over $[0,\infty)$ (resp. $(-\infty,0]$).

\end{prop}

\begin{proof} We see from 
\be \textup{Re}\{h\}=\frac{1}{2}t^{2}\cos 2\theta_{0}+\frac{1}{4}\log(1+t^{4}-2t^{2}\cos 2\theta_{0})\ee
that  \be \frac{d}{dt} \textup{Re}\{h(te^{i\theta_0})\} =\frac{t^3}{1+t^{4}-2t^{2}\cos 2\theta_{0}}  (t^{2}\cos 2\theta_{0}+1-2\cos^{2} 2\theta_{0})>0,  \quad \forall t>0,
\ee
for any given  $|\theta_0| \in [ \frac{\pi}{8},  \frac{\pi}{4}]$.  Part (i) then follows.  

For part (ii),    the monotonicity  follows from the simple  fact    $\frac{d}{dt} \textup{Re}\{h(te^{i\theta_0})\}>0, \forall t \in [0,t_{max})$.  Let $s=\cos 2\theta_0$, simple calculation shows 
\be t_{max} \cos\theta_0\geq 1\Longleftrightarrow  \left(1-\frac{1}{s}\right)\left(s+\frac{2+\sqrt{2}}{2}\right) \left(s+\frac{2-\sqrt{2}}{2}\right)\geq 0, \ee 
from which and  the assumption we complete part (ii).

Note that 
 \be \frac{d}{dx} \textup{Re}\{h(x+iy)\} =x+\frac{1}{2}\frac{x+1}{(x+1)^{2}+y^2}  +\frac{1}{2}\frac{x-1}{(x-1)^{2}+y^2}>0, \quad \forall x\geq 1,
\ee
 when $y\neq 0$ and 
 \be \frac{d}{dy} \textup{Re}\{h(x+iy)\} =-y+\frac{1}{2}\frac{y}{(x+1)^{2}+y^2}  +\frac{1}{2}\frac{y}{(x-1)^{2}+y^2}<0, \quad \forall y>0
\ee
whenever $x\geq 2$, we prove part (iii) and part (iv).
  \end{proof}

\section{Limiting eigenvalue density}
\label{sect:density}

The study of  scaling limits at the origin  investigated in the previous section introduces a scale in which the average spacing between eigenvalues is of order unity. A very different, but still well-defined, limiting process is the so-called limiting spectral measure in the   global scaling regime.  Usually, it has a density $\rho(x)$ with compact support $I\subset\mathbb R$ such that $\int_I \rho(x)dx=1$. Here $\rho(x)$   is referred to as the global density or limiting eigenvalue density.

For  squared singular values of  the product of independent  Ginibre matrices, i.e.,  eigenvalues of $W_M$ defined by \eqref{W1} but with $H=I_n$,  the global  limit corresponds to a change of variables $x_j\mapsto n^{M} x_j$ and  the global density is known to be the so-called Fuss--Catalan density with parameter $M$. Its $k$-th moment  ($k=0,1,\ldots$) is specified by the  Fuss--Catalan number 
\begin{equation}
\text{FC}_M(k)=\frac1{Mk+1}\binom{(M+1)k}k,
\end{equation}
see  e.g. \cite{BBCC11,NS06}.   The Catalan numbers  are the case $M=1$, corresponding to  the moments  of the  Marchenko-Pastur law in  a special case  and also  the even moments of the famous Wigner semicircle law (its odd moments vanishing).

Recently, Forrester, Ipsen and the author \cite{FIL17} turn to  the product $W_M$ in \eqref{W1} but with $H$ being  a GUE matrix,   i.e. $B=0$ in \eqref{meanGUE}.  After the change of variables $x_j\mapsto \frac{1}{\sqrt{2}}  n^{M+\frac{1}{2}} x_{j}$, they  prove that the global density  is an even function and its even moments are given by the  Fuss--Catalan numbers with parameter $2M+1$.  In this section  we investigate the global density for   the product matrix  $W_M$ with source $B$. 

Specifically, we assume that $n$ is even and 
 $b_{1}=\cdots=b_{n/2}=-b_{1+n/2}=\cdots=-b_{n}=\sqrt{n/2}a, \quad a\geq 0.$   To obtain the global density,  we need to make  the change of variables $x_j\mapsto \frac{1}{\sqrt{2}}  n^{M+\frac{1}{2}} x_{j}$.  To see  this, we may use free probability techniques; see e.g. \cite{NS06}.   
Suppose that  two selfajoint  non-commutative random variables $h$ and $w$ are free, and  at least one,  say,    $w$ is positive.   Recall that  the Stieltjes transform of $h$ with distribution $\mu$ is defined by 
\be G_{h}(z)=\int \frac{d \mu(x)}{z-x}, \qquad \textup{Im}(z)>0.\ee
 Let $S_{w}(z)$ denote the  $S$-transform of $w$, see e.g. \cite{NS06} for definition. If  $G_{h}(z)$ satisfies a functional equation $P(z,G_{h}(z))=0$, then 
we know from ~\cite{NS06} that   the Stieltjes transform $G_{hw}(z)$ of the product $hw$ satisfies
\begin{equation}\label{functional-recursion}
P\Big(zS_{w}(zG_{hw}(z)-1),\frac{zG_{hw}(z)}{S_{w}(zG_{hw}(z)-1)}\Big)=0.
\end{equation}
Moreover, we know that if $h$ is  a free convolution of   the standard semicircular law and   $\frac{1}{2}\big(\delta_{a}+\delta_{-a}\big)$ and if $w$ is given by the free Poisson distribution with parameter 1 (i.e. Mar\v cenko--Pastur law), then  $S_w(z)=1/(1+z)$ and $G_h$ satisfies the cubic equation 
\begin{equation} (G_{h}-z)^2 G_{h}+(1-a^{2}) G_{h}-z=0.
\end{equation}
Thus, using~\eqref{functional-recursion} $M$ times, we see that  the Stieltjes transform of limiting spectral measure for our product~\eqref{W1} indeed satisfies    a   functional equation 
\be  \big(z^{2M-1}g^{2M+1}-1\big)^{2}zg+(1-a^2)z^{2M-1}g^{2M+1}-1=0. \label{algebraiceq}\ee

Considering two special cases of \eqref{algebraiceq}, we can give explicit forms  of the limiting   eigenvalue densities  denoted by $\rho(a;x)$ and further compare leading  asymptotic behaviour near the origin. 
 
{\bf{Case 1: $a=0$.} }   Let 
\begin{equation}
  x^{2} = \frac{\big(\sin(( 2M+ 2)\varphi)\big)^{2M + 2}}{\sin \varphi \, \big(\sin(( 2M+ 1)\varphi) \big)^{ 2M+ 1}}, \qquad  -\frac{\pi}{2M+2}\leq \varphi\leq  \frac{\pi}{2M+2},
\end{equation}
\eqref{algebraiceq} has two special solutions 
\begin{equation}
  xg_{\pm} = \frac{\sin(( 2M+ 2)\varphi)}{\sin(( 2M+ 1)\varphi)}e^{\pm i\varphi}
\end{equation}
from which the density reads 
\be \rho(0;x)
=\frac{1}{\pi }    \sqrt{\frac{ \sin\varphi}{\sin(2M+ 1)\varphi}} \left(\frac{ \sin(2M+1)\varphi}{\sin(2M+ 2)\varphi}\right)^{ M} \sin\varphi, \quad  -\frac{\pi}{2M+2}\leq \varphi\leq  \frac{\pi}{2M+2}.
\ee
Moreover,  as $ x \rightarrow 0$ we have the leading term
\be \rho(0;x)\sim  \frac{1}{\pi } \sin\frac{\pi}{2M+2}\, |x|^{-1+\frac{1}{M+1}}.\ee
These  results have been obtained  in \cite{FIL17}.

{\bf{Case 2: $a=1$.} }  In this case \eqref{algebraiceq}  reduces to
\be  \big(z^{2M-1}g^{2M+1}-1\big)^{2}zg-1=0. \label{algebraiceq2}\ee
 Let 
\begin{equation}
  x^{2} = \frac{\big(\sin(( 4M+ 3)\varphi)\big)^{4M+ 3}}{\sin \varphi \, \big(\sin(( 4M+ 2)\varphi) \big)^{ 4M+ 2}}, \qquad  -\frac{\pi}{4M+3}\leq \varphi\leq  \frac{\pi}{4M+3},
\end{equation}
\eqref{algebraiceq2} has two special solutions 
\begin{equation}
  xg_{\pm} = \left(\frac{ \sin(4M+3)\varphi}{\sin(4M+ 2)\varphi}\right)^{ 2}e^{\pm 2 i\varphi},
\end{equation}
from which the density reads 
\be \rho(1;x)
=\frac{1}{\pi }    \sqrt{\frac{ \sin\varphi}{\sin(4M+ 3)\varphi}} \left(\frac{ \sin(4M+2)\varphi}{\sin(4M+ 3)\varphi}\right)^{2 M-1} \sin2\varphi,  \ -\frac{\pi}{4M+3}\leq \varphi\leq  \frac{\pi}{4M+3}.\ee
Moreover,  as $ x \rightarrow 0$ we have the leading term
\be \rho(1;x)\sim  \frac{1}{\pi } \sin\frac{2\pi}{4M+3}\, |x|^{-1+\frac{4}{4M+3}}.\ee

Generally, we expect from the algebraic equation  \eqref{algebraiceq}  that there   exist exactly three families of  blow-up exponents at the origin  for  $\rho(a;x)$, which reads  as $x \rightarrow 0$
\be \rho(a;x)\sim \begin{cases} c_{a}\, |x|^{-1+\frac{1}{M+1}},  & \quad 0\leq a <1;\\
c_a \, |x|^{-1+\frac{4}{4M+3}},  & \quad a=1;\\
c_{a}\, |x|^{-1+\frac{2}{2M+1}},  & \quad a>1.\end{cases}\ee
If so, this will be consistent with the local scalings chosen in Theorem \ref{transition}. 

Finally, we stress that the above parametrization representations  are of vital importance in proving the sine kernel in the bulk, see e.g.  \cite{LWZ14} for more details. 

\paragraph{Acknowledgements} 
This work was partially supported   by   ERC Advanced Grant No. 338804,  the Youth Innovation Promotion Association CAS  \#2017491, the Fundamental Research Funds for the Central Universities Grants WK0010450002 and WK3470000008,  and  Anhui Provincial Natural Science Foundation \#1708085QA03.   The author would particularly like to thank L. Erd\H{o}s  for support and P.J. Forrester  for  useful discussions.


\begin{thebibliography}{99}

\bibitem{AM07}
 M. Adler, M., P. van Moerbeke,  \emph{PDEs for the Gaussian ensemble with external source and the Pearcey
distribution}. Commun. Pure Appl. Math. 60 (2007), 1261--1292.
\bibitem{ABF11}
G. Akemann, J. Baik, and P. Di Francesco (eds.),
\emph{The Oxford handbook of random matrix theory.} 
Oxford University Press, 2011.
\bibitem{AI15}
G. Akemann, and J. R. Ipsen,
\emph{Recent exact and asymptotic results for products of independent random matrices.}
Acta Physica Polonica B 46 (2015), no. 9, 1747--1784.

\bibitem{AIK13}
G. Akemann, J. R. Ipsen, and M. Kieburg, 
\emph{Products of rectangular random matrices: singular values and progressive scattering.} 
Phys. Rev. E 88 (2013), 052118.

\bibitem{AKW13}
G. Akemann, M. Kieburg, and L. Wei,
\emph{Singular value correlation functions for products of Wishart random matrices.}
J. Phys. A  46 (2013), 275205.

\bibitem{AGT10}
N. Alexeev, F. G\"otze, and A. Tikhomirov, \emph{Asymptotic distribution of singular values of powers of random matrices.} Lith. Math. J. 50 (2010), 121--132.
%
%




\bibitem{ABK05} A. I. Aptekarev, P. M. Bleher,  and A. B.J  Kuijlaars, \emph{
Large n Limit of Gaussian Random Matrices with External Source, Part II}. Commun. Math. Phys. 259 (2005),  367--389. 

\bibitem{BBP}  J. Baik,  G.    Ben Arous and  S. P\'{e}ch\'{e},  \emph{Phase transition of the largest eigenvalue for non-null complex
sample covariance matrices}. Ann. Prob. 33 (2005), no. 5, 1643--1697.

\bibitem{BBCC11}
T. Banica, S. T. Belinschi, M. Capitaine, and B. Collins, 
\emph{Free Bessel laws.} 
Canad. J. Math. 63 (2011),  3--37.

\bibitem{BB15} M. Bertola,   T.  Bothner,   \emph{Universality conjecture and results for a model
of several coupled positive-definite matrices}.   Commun. Math. Phys. 337 (2015), 1077--1141.
\bibitem{BGS14} M. Bertola,   M. Gekhtman,  and   J. Szmigielski,   \emph{Cauchy-Laguerre two-matrix
model and the Meijer-G random point field}. Commun. Math. Phys. 326 (2014), 111--144.


\bibitem{CP16} 
M. Capitaine,  S. P\'{e}ch\'{e},  \emph{Fluctuations at the edges of the spectrum of the full rank deformed GUE}.  Probab. Theory Related Fields 165 (2016), 117--161. 


\bibitem{BK04}
P. M.  Bleher, A.B.J. Kuijlaars, \emph{
Large n Limit of Gaussian Random Matrices with External Source, Part I.}  Commun. Math. Phys. 252 (2004),  43--76. 


\bibitem{BK07}
P. M.  Bleher, A.B.J. Kuijlaars, \emph{Large n limit of Gaussian random matrices with external
source, part III: double scaling limit.}  Commun. Math. Phys. 270 (2007), 481--517.



\bibitem{Bo99}
A. Borodin, 
\emph{Biorthogonal ensembles.} 
Nucl. Phys. B 536 (1999),  704--732.



\bibitem{BH96}
E. Br\'ezin, S. Hikami, \emph{Correlations of nearby levels induced by a random potential}. Nuclear Phys. B
479 (1996), 697--706. 

\bibitem{BH98}
E. Br\'ezin, S. Hikami, \emph{Universal singularity at the closure of a gap in a random matrix
theory}. Phys. Rev. E 57 (1998), 4140--4149.

\bibitem{BJLNS10}
Z. Burda, A. Jarosz, G. Livan, M. A. Nowak, and A. Swiech,
\emph{Eigenvalues and singular values of products of rectangular Gaussian random matrices.} 
Phys. Rev. E 82 (2010), 061114.


\bibitem{CKW15}
T. Claeys, A. B. J. Kuijlaars, and D. Wang,
\emph{Correlation kernels for sums and products of random matrices.}
Random Matrices: Theory Appl. 4 (2015), 1550017.


\bibitem{CW14}
T. Claeys, D. Wang,
\emph{ Random matrices with equispaced external source}.  Comm.
Math. Phys. 328 (2014), 1023--1077.

%

\bibitem{Fo10}
P. J. Forrester, 
\emph{Log-gases and random matrices}.
Princeton University Press, Princeton, NJ, 2010.

\bibitem{Fo14}
P. J. Forrester,
\emph{Eigenvalue statistics for for product complex Wishart matrices}. 
J. Phys. A 47 (2014), 345202.
%


\bibitem{FIL17}
P. J. Forrester, J. R. Ipsen, and D.-Z. Liu,
\emph{Matrix product ensembles of Hermite-type}. Preprint 
arXiv:1702.07100.
%


\bibitem{FL16}
P. J. Forrester, and D.-Z. Liu,
\emph{Singular values for products of complex Ginibre matrices with a source: hard edge limit and phase transition}. Commun. Math. Phys. 344 (2016), 333--368. 


%




%
%
%

\bibitem{HC57}
Harish-Chandra,
\emph{Differential operators on a semisimple Lie algebra.} 
Amer. J. Math. 79(1957), 87--120.


\bibitem{IK14}
J. R. Ipsen, and M. Kieburg,
\emph{Weak commutation relations and eigenvalue statistics for products of rectangular random matrices}. 
Phys. Rev. E 89 (2014), 032106.

\bibitem{IZ80}
C. Itzykson, and J.B. Zuber,
\emph{The planar approximation II.}
J. Math. Phys. 21 (1980), 411-421.

\bibitem{Jo01}
K. Johansson,  \emph{Universality of the local spacing distribution in certain ensembles of Hermitian Wigner
matrices}. Commun. Math. Phys. 215 (2001), 683--705.


\bibitem{KK16}
M. Kieburg, and H. K\"osters,
\emph{Exact Relation between Singular Value and Eigenvalue Statistics}.
Random Matrices: Theor. Appl. 5 (2016), 1650015.

\bibitem{KKS15}
M. Kieburg, A. B. J. Kuijlaars, and D. Stivigny, 
\emph{Singular value statistics of matrix products with truncated unitary matrices.} 
Int. Math. Res. Not. IMRN 2016 (2016),  no.11, 3392--3424.


\bibitem{KY14}
A. Knowles,  J. Yin, \emph{The outliers of a deformed Wigner matrix}. Ann. Probab. 42 (2014), 1980--2031.

\bibitem{Ku16}
A. B. J. Kuijlaars,
\emph{Transformations of polynomial ensembles.} 
Contemporary Mathematics,  Vol. 661,     
Amer. Math. Soc., Providence, RI, 2016. 


 
\bibitem{KS14}
A. B. J. Kuijlaars, and D. Stivigny, 
\emph{Singular values of products of random matrices and polynomial ensembles.} 
Random Matrices: Theor. Appl. 3 (2014), 1450011.

\bibitem{KZ14}
A. B. J. Kuijlaars, and L. Zhang, 
\emph{Singular values of products of Ginibre random matrices, multiple orthogonal polynomials and hard edge scaling limits.} 
Commun. Math. Phys. 332 (2014),  759--781.




\bibitem{LS16} J. O. Lee, K.  Schnelli, 
\emph{Extremal eigenvalues and eigenvectors of deformed Wigner matrices}.  Probab.  Theory and Related Fields 164 (2016), 165--241.

\bibitem{LSSY} J. O. Lee, K.  Schnelli, B.  Stetler, and H.-T.  Yau
\emph{Bulk universality for deformed Wigner matrices.}  Ann. Probab. 44 (2016), 2349--2425.







\bibitem{LWZ14}
D.-Z. Liu, D. Wang, and L. Zhang,
\emph{Bulk and soft-edge universality for singular values of products of Ginibre random matrices.} 
Ann. Inst. H. Poincar\'{e} Probab. Stat. 52 (2016), 1734--1762.



%

\bibitem{NS06}
A. Nica, and R. Speicher,
\emph{Lectures on the combinatorics of free probability}.
Cambridge University Press, 2006.


\bibitem{OR07} 
A. Okounkov, N. Reshetikhin, \emph{Random skew plane partitions and the Pearcey process}. 
Comm. Math. Phys. 269 (2007), 571--609.


\bibitem{Ol10}
Olver, F.W.J., Lozier, D.W., Boisvert, R.F., Clark, C.W. (eds.): \textit{NIST Handbook of Mathematical Functions},
Cambridge University Press, Cambridge, 2010 (Print companion to [DLMF]).



%

\bibitem{Pe06} S. P\'{e}ch\'{e},  \emph{The largest eigenvalue of small rank perturbations of Hermitian random matrices}. Probab. Theory Related Fields 134 (2006),  127--173.

\bibitem{Sh11}
 T. Shcherbina, \emph{On universality of local edge regime for the deformed Gaussian unitary ensemble}. J.
Stat. Phys. 143 (2011), 455--481.



\bibitem{TW06} 
C. Tracy, H. Widom,  \emph{The Pearcey process}. Commun. Math. Phys. 263 (2006), 381--400.



  
\end{thebibliography}
\end{document}